\UseRawInputEncoding
\documentclass[12pt]{article}
\usepackage[centertags]{amsmath}
\usepackage{amsfonts}
\usepackage{amssymb}
\usepackage{latexsym}
\usepackage{amsthm}
\usepackage{newlfont}
\usepackage{graphicx}
\usepackage{listings}
\usepackage{booktabs}
\usepackage{abstract}
\usepackage{enumerate}
\usepackage{xcolor}
\RequirePackage{srcltx}
\date{}

\newlength{\defbaselineskip}
\newcommand{\setlinespacing}[1]%
           {\setlength{\baselineskip}{#1 \defbaselineskip}}

\newcommand{\actaqed}{\hfill $\actabox$}
{\medskip\noindent \textit{Proof of #1. }}%
{\actaqed \medskip}

\def\cA{{\mathcal A}}

\def\cC{{\mathcal C}}

\def\cG{{\mathcal G}}

\def\cK{{\mathcal K}}

\def\cM{{\mathcal M}}

\def\cT{{\mathcal T}}

\def\cV{{\mathcal V}}

\def\bbC{{\mathbb C}}

\def\bbN{{\mathbb N}}

\def\bbR{{\mathbb R}}

\def\bbT{{\mathbb T}}

\def\bbZ{{\mathbb Z}}

\def\bA{{\mathbf A}}

\def\bF{{\mathbf F}}

\def\bH{{\mathbf H}}

\def\bN{{\mathbf N}}

\def\bW{{\mathbf W}}

\def\bj{\mathbf j}
\def\bk{\mathbf k}

\def\bs{\mathbf s}
\def\bt{\mathbf t}
\def\btt{\mathbf t}
\def\bu{\mathbf u}
\def\bv{\mathbf v}
\def\bw{\mathbf w}
\def\bx{\mathbf x}
\def\by{\mathbf y}

 \def \<{\langle}
\def\>{\rangle}

\def \Og{\Omega}

\def \ff{\varphi}

\def\bt{\beta}

\def \sign{\operatorname{sign}}

\def\bt{\beta}

\newtheorem{Theorem}{Theorem}[section]
\newtheorem{Lemma}{Lemma}[section]
\newtheorem{Definition}{Definition}[section]
\newtheorem{Proposition}{Proposition}[section]
\newtheorem{Remark}{Remark}[section]

\newtheorem{Corollary}{Corollary}[section]
\numberwithin{equation}{section}

\newcommand{\be}{\begin{equation}}
\newcommand{\ee}{\end{equation}}

\begin{document}

\title{Some lower bounds for optimal sampling recovery of functions with mixed smoothness}

\author{A. Gasnikov and     V. Temlyakov }

\newcommand{\Addresses}{{
  \bigskip
  \footnotesize

\medskip
A.V. Gasnikov, \\ \textsc{Ivannikov institute for System Programming of Russian Academy of Sciences, Moscow, Russia;\\ Steklov Mathematical Institute of Russian Academy of Sciences, Moscow, Russia;  \\ Innopolis University, Tatarstan, Russia.
\\ E-mail:} \texttt{gasnikov@yandex.ru}

 \medskip
  V.N. Temlyakov, \textsc{University of South Carolina, USA,\\ Steklov Mathematical Institute of Russian Academy of Sciences, Russia;\\ Lomonosov Moscow State University, Russia; \\ Moscow Center of Fundamental and Applied Mathematics, Russia.\\  
E-mail:} \texttt{temlyakovv@gmail.com}

}}

\maketitle

\begin{abstract}{Recently, there was a substantial progress in the problem of sampling recovery on function classes 
with mixed smoothness. Mostly, it has been done by proving new and sometimes optimal upper bounds for both linear sampling recovery and for nonlinear sampling recovery. In this paper we address the problem of lower 
bounds for the optimal rates of nonlinear sampling recovery. 
  In the case of linear recovery one can use the well developed theory of estimating the Kolmogorov and linear widths for establishing some lower bounds for the optimal rates.  In the case of nonlinear recovery we cannot use the above approach. It seems like the only technique, which is available now, is based on some simple observations. We demonstrate how these observations can be used. }
\end{abstract}

\section{Introduction}
\label{In}

Recently, there was a substantial progress in the problem of sampling recovery on function classes 
with mixed smoothness. This paper is a followup of the recent papers 
\cite{KoTe}, \cite{VT202}, \cite{VT203}, and \cite{VT205}. In this paper we address the problem of lower 
bounds for the optimal rates of sampling recovery. The problem of sampling recovery on function classes 
with mixed smoothness has a long history with first results going back to the 1963 (see \cite{Smol}). In many cases this 
problem is still open. We refer the reader to the books \cite{DTU} and \cite{VTbookMA} 
for the corresponding historical discussion.

In this section we describe the problem setting and present some known upper bounds. In Sections \ref{qp} -- \ref{S} 
we obtain some new results. Section \ref{D} contains a discussion. In this paper we admit the following convenient and standard notation agreement. 
We use $C$, $C'$ and $c$, $c'$ to denote various positive constants. Their arguments indicate the parameters, which they may depend on. Normally, these constants do not depend on a function $f$ and running parameters $m$, $v$, $u$. We use the following symbols for brevity. For two nonnegative sequences $a=\{a_n\}_{n=1}^\infty$ and $b=\{b_n\}_{n=1}^\infty$ the relation $a_n\ll b_n$ means that there is  a number $C(a,b)$ such that for all $n$ we have $a_n\le C(a,b)b_n$. Relation $a_n\gg b_n$ means that 
 $b_n\ll a_n$ and $a_n\asymp b_n$ means that $a_n\ll b_n$ and $a_n \gg b_n$. 
 For a real number $x$ denote $[x]$ the integer part of $x$, $\lceil x\rceil$ -- the smallest integer, which is 
 greater than or equal to $x$. 
 
 We study the multivariate periodic functions defined on $\bbT^d := [0,2\pi]^d$. Denote for $1\le p<\infty$
 $$
 \|f\|_p := \left((2\pi)^{-d}\int_{\bbT^d}  |f(\bx)|^pd\bx\right)^{1/p},\qquad \bx= (x_1,\dots,x_d)
 $$
 and for $p=\infty$
 $$
 \|f\|_\infty := \sup_{\bx\in\bbT^d} |f(\bx)|.
 $$

We begin with the definition of classes $\bW^r_q$.
\begin{Definition}\label{InD1}
In the univariate case, for $r>0$, let
\be\label{sr7}
F_r(x):= 1+2\sum_{k=1}^\infty k^{-r}\cos (kx-r\pi/2)
\ee
and in the multivariate case, for $\bx=(x_1,\dots,x_d)$, let
$$
F_r(\bx) := \prod_{j=1}^d F_r(x_j).
$$
Denote
$$
\bW^r_q := \{f:f=\varphi\ast F_r,\quad \|\varphi\|_q \le 1\},
$$
where
$$
(\varphi \ast F_r)(\bx):= (2\pi)^{-d}\int_{\bbT^d} \ff(\by)F_r(\bx-\by)d\by.
$$
\end{Definition}
The classes $\bW^r_q$ are classical classes of functions with {\it dominated mixed derivative}
(Sobolev-type classes of functions with mixed smoothness).
The reader can find results on approximation properties of these classes in the books \cite{VTbookMA} and \cite{DTU}.

 We now proceed to the definition of the classes $\bH^r_p$, which is based on the mixed differences. In this paper we obtain new results for these classes.
 
\begin{Definition}\label{InD2}
Let  $\btt =(t_1,\dots,t_d )$ and $\Delta_{\btt}^l f(\bx)$
be the mixed $l$-th difference with
step $t_j$ in the variable $x_j$, that is
$$
\Delta_{\btt}^l f(\bx) :=\Delta_{t_d,d}^l\dots\Delta_{t_1,1}^l
f(x_1,\dots ,x_d ) .
$$
Let $e$ be a subset of natural numbers in $[1,d ]$. We denote
$$
\Delta_{\btt}^l (e) :=\prod_{j\in e}\Delta_{t_j,j}^l,\qquad
\Delta_{\btt}^l (\varnothing) := Id \,-\, \text{identity operator}.
$$
We define the class $\bH_{p,l}^r B$, $l > r$, as the set of
$f\in L_p$ such that for any $e$
\be\label{B7a}
\bigl\|\Delta_{\btt}^l(e)f(\bx)\bigr\|_p\le B
\prod_{j\in e} |t_j |^r .
\ee
In the case $B=1$ we omit it. It is known (see, for instance, \cite{VTbookMA}, p.137) that the classes $\bH^r_{p,l}$ with different $l>r$ are equivalent. So, for convenience we fix one $l= [r]+1$ and omit $l$ from the notation. 
\end{Definition}

It is well known that in the univariate case ($d=1$) the approximation properties of the above $\bW$ and $\bH$ classes 
are similar. It is also well known that in the multivariate case ($d\ge 2$) asymptotic characteristics (for instance, Kolmogorov widths, entropy numbers, best hyperbolic cross trigonometric approximations and others) have different rate of decay in 
the majority of cases. Recently, a new scale of classes has been introduced and studied. It turns out that this scale is 
convenient for simultaneous analysis of optimal sampling recovery of both the $\bW$ and the $\bH$ classes. We give a corresponding definitions now. Let $\mathbf s=(s_1,\dots,s_d )$ be a  vector  whose  coordinates  are
nonnegative integers
$$
\rho(\mathbf s) := \bigl\{ \mathbf k\in\mathbb Z^d:[ 2^{s_j-1}] \le
|k_j| < 2^{s_j},\qquad j=1,\dots,d \bigr\},
$$
where $[a]$ means the integer part of a real number $a$.  
For $f\in L_1 (\bbT^d)$
$$
\delta_{\mathbf s} (f,\mathbf x) :=\sum_{\mathbf k\in\rho(\mathbf s)}
\hat f(\mathbf k)e^{i(\mathbf k,\mathbf x)},\quad \hat f(\mathbf k) := (2\pi)^{-d}\int_{\bbT^d} f(\bx)e^{-i(\mathbf k,\mathbf x)}d\bx.
$$
\begin{Definition}\label{InD3}
Consider functions with absolutely convergent Fourier series. For such functions define the Wiener norm (the $A$-norm or the $\ell_1$-norm)
$$
\|f\|_A := \sum_{\mathbf k}|\hat f({\mathbf k})|.
$$
The following classes, which are convenient in studying sparse approximation with respect to the trigonometric system,
were introduced and studied in \cite{VT150}. Define for $f\in L_1(\bbT^d)$
$$
f_j:=\sum_{\|\bs\|_1=j}\delta_\bs(f), \quad j\in \bbN_0,\quad \bbN_0:=\bbN\cup \{0\}.
$$
For parameters $ a\in \bbR_+$, $ b\in \bbR$ define the class
$$
\bW^{a,b}_A:=\{f: \|f_j\|_A \le 2^{-aj}(\bar j)^{(d-1)b},\quad \bar j:=\max(j,1), \quad j\in \bbN\}.
$$
\end{Definition}

The following embedding result follows easily from known results. We give a detailed proof of it in Section \ref{P}.

\begin{Proposition}\label{InP1} We have for $r>1/q$ 
\be\label{In1}
\bW^r_q \hookrightarrow \bW^{a,b}_A \quad \text{with} \quad a=r-1/q,\, b=1-1/q,\quad 1<q\le 2;
\ee
\be\label{In2}
\bH^r_q \hookrightarrow \bW^{a,b}_A \quad \text{with} \quad a=r-1/q,\, b=1,\quad 1\le q\le 2.
\ee
\end{Proposition}

We give a very brief history of the recent development of the sampling recovery on these 
classes. We refer the reader   to the books \cite{DTU} and \cite{VTbookMA} for the previous results. In this paper we study the following characteristic of the optimal sampling recovery. Let $\Omega$ be a compact subset of $\bbR^d$ with the probability measure $\mu$ on it.
For a function class $W\subset \cC(\Omega)$,  we  define (see \cite{TWW})
$$
\varrho_m^o(W,L_p) := \inf_{\xi } \inf_{\cM} \sup_{f\in W}\|f-\cM(f(\xi^1),\dots,f(\xi^m))\|_p,
$$
where $\cM$ ranges over all  mappings $\cM : \bbC^m \to   L_p(\Omega,\mu)$  and
$\xi$ ranges over all subsets $\{\xi^1,\cdots,\xi^m\}$ of $m$ points in $\Og$.
Here, we use the index {\it o} to mean optimality. Clearly, the above characteristic is a characteristic of nonlinear recovery.
For a discussion of the sampling recovery by linear methods see Section \ref{D}. 

The authors of \cite{JUV} (see Corollary 4.16 in v3) proved the following bound for $1<q<2$, $r>1/q$ and  $m\ge c(r,d,q) v(\log(2v))^3$, $v\in \bbN$,
\begin{equation}\label{H6}
\varrho_{m}^{o}(\bW^{r}_q,L_2(\bbT^d)) \le C(r,d,q)  v^{-r+1/q-1/2} (\log (2v))^{(d-1)(r+1-2/q)+1/2}.     
\end{equation}
The authors of  \cite{DTM2} proved the following bound
\be\label{H7}
\varrho_{m}^{o}(\bW^{r}_q,L_2(\bbT^d)) \le C'(r,d,q)  v^{-r+1/q-1/2} (\log (2v))^{(d-1)(r+1-2/q)}      
\end{equation}
provided that
\be\label{H8}
m\ge c'(r,d,q) v(\log(2v))^3.
\ee

In the above mentioned results the sampling recovery in the $L_2$ norm has been studied. The technique, which was used in the proofs of the bounds (\ref{H6}) and (\ref{H7}) is heavily based on the fact that we approximate in the $L_2$ norm.  The following upper bound was proved in \cite{KoTe}.

 \begin{Theorem}[{\cite{KoTe}}]\label{HT2}  Let $1<q\le 2\le p<\infty$ and $r>1/q$.  There exist two constants $c=c(r,d,p,q)$ and $C=C(r,d,p,q)$ such that     we have the bound   
 \begin{equation}\label{H9}
  \varrho_{m}^{o}(\bW^{r}_q,L_p(\bbT^d)) \le C   v^{-r+1/q-1/p} (\log (2v))^{(d-1)(r+1-2/q)}
\end{equation}
   for any $v\in\bbN$ and any $m$ satisfying
$$
m\ge c  v(\log(2v))^3.
$$
\end{Theorem}

Thus, Theorem \ref{HT2} extends the previously known upper bound for $p=2$ to the case $p\in [2,\infty)$. 

In \cite{KoTe} Theorem \ref{HT2} was derived from the embedding (\ref{In1}) and the following result for the $\bW^{a,b}_A$ classes (see \cite{KoTe}, Theorem 5.3, Remark 5.1, and Proposition 5.1). 

\begin{Theorem}[{\cite{KoTe}}]\label{srT2}  Let $p\in [2,\infty)$. There exist two constants $c(a,p,d)$ and $C(a,b,p,d)$
such that we have the bound
\begin{equation}\label{sr3}
 \varrho_{m}^{o}(\bW^{a,b}_A,L_p(\bbT^d)) \le C(a,b,p,d)  v^{-a-1/p} (\log (2v))^{(d-1)(a+b)}
\end{equation}
    for any $v\in\bbN$ and any $m$ satisfying
$$
m\ge c(a,d,p) v(\log(2v))^3.
$$
\end{Theorem}

Theorem \ref{srT2} and embedding (\ref{In2}) imply the following analog of the bound (\ref{H9}) for the $\bH$ classes. There exist two constants $c=c(r,d,p,q)$ and $C=C(r,d,p,q)$ such that     we have the bound for $r>1/q$
 \begin{equation}\label{In6}
  \varrho_{m}^{o}(\bH^{r}_q,L_p(\bbT^d)) \le C   v^{-r+1/q-1/p} (\log (2v))^{(d-1)(r+1-1/q)}
\end{equation}
   for any $v\in\bbN$ and any $m$ satisfying
$$
m\ge c  v(\log(2v))^3.
$$
However, we point out that the bound (\ref{In6}) is weaker than the corresponding known bound for the linear recovery 
(see Section \ref{D} for a discussion).

The following lower bound for the $\bH$ classes is the main result of this paper.  

\begin{Theorem}\label{qpT1} For  $1\le q\le p< \infty$, $p>1$, $r>1/q$, we have  
$$
\varrho_m^o(\bH^r_q,L_p) \ge c(d)m^{-r +1/q-1/p}  (\log m)^{(d-1)/p}  .
$$
\end{Theorem}

Theorem \ref{qpT1} is a direct corollary of Lemma \ref{qpL2}, which is proved in Section \ref{qp}. The reader can find 
a discussion of this result in Sections \ref{qp} and \ref{D}. Note that a new nontrivial feature of Theorem \ref{qpT1} is the logarithmic factor $(\log m)^{(d-1)/p}$, which shows that some logarithmic in $m$ factor is needed. 

In Section \ref{q1} we derive the following lower bound from the known results developed in numerical integration. 

\begin{Proposition}\label{q1P2} We have for $r>0$
$$
\varrho_m^o(\bH^r_\infty,L_1) \gg m^{-r}(\log m)^{d-1}.
$$
\end{Proposition}

In Section \ref{S} we formulate the setting of the sampling recovery in the general space $L_p(\Omega,\mu)$, $1\le p<\infty$, and instead of the trigonometric system $\cT^d$ we study a general uniformly bounded system $\Psi = \{\psi_{\bk}\}_{\bk\in \bbZ^d}$. We prove there a lower bound for the new classes defined with respect to the trigonometric system $\cT^d$. 

\section{Preliminaries}
\label{P}

 We need some classical trigonometric polynomials. The univariate Fej\'er kernel of order $j - 1$:
$$
\mathcal K_{j} (x) := \sum_{|k|\le j} \bigl(1 - |k|/j\bigr) e^{ikx} 
=\frac{(\sin (jx/2))^2}{j (\sin (x/2))^2}.
$$
The Fej\'er kernel is an even nonnegative trigonometric
polynomial of order $j-1$.  It satisfies the obvious relations
\be\label{FKm}
\| \mathcal K_{j} \|_1 = 1, \qquad \| \mathcal K_{j} \|_{\infty} = j.
\ee
Let $\cK_\bj(\bx):= \prod_{i=1}^d \cK_{j_i}(x_i)$ be the $d$-variate Fej\'er kernels for $\bj = (j_1,\dots,j_d)$ and $\bx=(x_1,\dots,x_d)$.

The univariate de la Vall\'ee Poussin kernels are defined as follows
$$
\cV_m := 2\cK_{2m} - \cK_m.
$$
We also need the following special trigonometric polynomials.
Let $s$ be a nonnegative integer. We define
$$
\mathcal A_0 (x) := 1, \quad \mathcal A_1 (x) := \mathcal V_1 (x) - 1, \quad
\mathcal A_s (x) := \mathcal V_{2^{s-1}} (x) -\mathcal V_{2^{s-2}} (x),
\quad s\ge 2,
$$
where $\mathcal V_m$ are the de la Vall\'ee Poussin kernels defined above.
For $\bs=(s_1,\dots,s_d)\in \bbN^d_0$ define
$$
\cA_\bs(\bx) := \prod_{j=1}^d  \cA_{s_j}(x_j),\qquad \bx=(x_1,\dots,x_d).
$$

We now prove Proposition \ref{InP1}. 

{\bf Proof of Proposition \ref{InP1}.} First, we prove (\ref{In1}). It is well known (see, for instance, \cite{VTmon}, Ch.2, Theorem 2.1) that for $f\in \bW^r_q$ one has for $1<q<\infty$
\be\label{C2}
\|f_j\|_q \le C(d,q,r)2^{-jr},\quad j\in \bbN.
\ee
The known results (see   Theorem \ref{A} below) imply for $1<q\le 2$
\be\label{C3}
\|f_j\|_A \le C(d,q)2^{j/q} j^{(d-1)(1-1/q)} \|f_j\|_q  
\ee
\be\label{C3'}
\le C(d,q,r) 2^{-(r-1/q)j}j^{(d-1)(1-1/q)}    .
\ee
Therefore, class $\bW^r_q$ is embedded into the class $\bW^{a,b}_A$ with $a=r-1/q$ and $b = 1-1/q$. 

Second, we prove (\ref{In2}). Here we need a well known result on the representation of the $\bH$ classes (see, for instance, \cite{VTbookMA}, p.137). 

\begin{Theorem}\label{PT1} Let $f\in \bH_{q,l}^r$.
Then for $\bs\ge\mathbf 0$
\be\label{s0.10}
\bigl\| A_{\bs} (f) \bigr\|_q \le C(r,d,l)2^{-r\|\bs\|_1}, \qquad
1 \le q \le \infty , 
\ee
\be\label{s0.11}
\bigl\|\delta_{\bs} (f) \bigr\|_q \le C(r,d,q,l)2^{-r\|\bs\|_1},
\qquad 1 < q < \infty . 
\ee
Conversely, from {\rm(\ref{s0.10})}  or {\rm(\ref{s0.11})} it follows that there exists a
$B > 0$, which does not depend on $f$, such that $f\in
\bH_{q,l}^r B$.
\end{Theorem}

By Theorem \ref{PT1} we obtain that for $f\in \bH^r_q$ one has for $1\le q\le \infty$
\be\label{In3}
\|A_\bs(f)\|_q \le C(d,r)2^{-r\|\bs\|_1},\quad \bs\in \bbN^d_0.
\ee
It is known and easy to see that for $q\in [1,2]$
\be\label{In4}
\|A_\bs(f)\|_A \le C(d)2^{\|\bs\|_1/q}\|A_\bs(f)\|_q \le C'(d,r)2^{-(r-1/q)\|\bs\|_1} .
\ee
Therefore,
\be\label{In5}
\|f_j\|_A\le C''(d,r) 2^{-(r-1/q)j}j^{(d-1)}    ,
\ee
which completes the proof of (\ref{In2}).

We formulate some known results from harmonic analysis and from the hyperbolic cross approximation theory, which will be used in our analysis.




We begin with the problem of estimating $\|f\|_u$ in terms of
the array $\bigl\{ \|\delta_{\bs} (f)\|_v  \bigr\}$.
Here and below in this section $u$ and $v$  are  scalars  such
that $1\le u,v\le \infty$. Let an array
$\varepsilon = \{\varepsilon_{\bs}\}$ be given, where
$\varepsilon_{\bs}\ge 0$, $\bs = (s_1 ,\dots,s_d)$,
and $s_j$ are nonnegative integers, $j = 1,\dots,d$.
We denote by $G(\varepsilon,v)$ and $F(\varepsilon,v)$
the following sets of functions
$(1\le v\le \infty)$:
$$
G(\varepsilon,v) := \bigl\{ f\in L_v  : \bigl\|\delta_{\bs} (f)
\bigr\|_v
\le\varepsilon_{\bs}\qquad\text{ for all }\bs\bigr\} ,
$$
$$
F(\varepsilon,v) := \bigl\{ f\in L_v  : \bigl\|\delta_{\bs} (f)
\bigr\|_v
\ge\varepsilon_{\bs}\qquad\text{ for all }\bs\bigr\}.
$$

The following theorem is from \cite{VTmon}, p.29 (see also \cite{VTbookMA}, p.94). For the special case $v=2$  see \cite{T29} and \cite{VTmon}, p.86. 

\begin{Theorem}\label{T1.1} The following relations hold:
\begin{equation}\label{1.1}
\sup_{f\in G(\varepsilon,v)}\|f\|_u \asymp\left(\sum_{\bs}
\varepsilon_{\bs}^u2^{\|\bs\|_1(u/v-1)}\right)^{1/u},
\qquad 1\le v < u < \infty ;
\end{equation}
\begin{equation}\label{1.2}
\inf_{f\in F(\varepsilon,v)}\|f\|_u \asymp\left(\sum_{\bs}
\varepsilon_{\bs}^u  2^{\|\bs\|_1(u/v-1)}\right)^{1/u},
\qquad 1< u < v\le\infty ,
\end{equation}
with constants independent of $\varepsilon$.
\end{Theorem}
We will need a corollary of Theorem \ref{T1.1} (see \cite{VTmon}, Ch.1, Theorem 2.2), which we formulate as a theorem.
Let $Q$ be a finite set of points in $\mathbb Z^d$, we denote
$$
\cT(Q) :=\left\{ t\colon t(\mathbf x) =\sum_{\mathbf k\in Q}a_{\mathbf k}
e^{i(\bk,\bx)}\right\} 
$$
and
$$
Q_n := \bigcup_{\bs:\|\bs\|_1 \le n}\rho(\bs).
$$
\begin{Theorem}\label{A} Let $1<q\le 2$. For any $t\in \cT(Q_n)$ we have
$$
\|t\|_A:=\sum_\bk |\hat t(\bk)| \le C(q,d) 2^{n/q} n^{(d-1)(1-1/q)}\|t\|_q.
$$
\end{Theorem}


\section{The case $1\le q\le p\le \infty$}
\label{qp}

Let us discuss lower bounds for the nonlinear characteristic $\varrho_m^o(W,L_p)$.
 Denote for $\bN = (N_1,\dots,N_d)$, $N_j\in \bbN_0$, $j=1,\dots,d$,
 $$
\Pi(\bN,d) := \{\bk \in \bbZ^d\,:\, |k_j|\le N_j, j=1,\dots,d\}
$$
and
$$
\cT(\bN,d) := \left\{f=\sum_{\bk\in \Pi(\bN,d)} c_\bk e^{i(\bk,\bx)}\right\},\quad \vartheta(\bN) :=\prod_{j=1}^d (2N_j+1). 
$$
In this section $\Omega = \bbT^d$ and $\mu$ is the normalized Lebesgue measure on $\bbT^d$. The following Lemma \ref{qpL1} was proved in \cite{VT203}.

\begin{Lemma}[{\cite{VT203}}]\label{qpL1} Let $1\le q\le p\le \infty$ and let  $\cT(\bN,d)_q$ denote the unit $L_q$-ball of the subspace $\cT(\bN,d)$. Then we have for $m\le \vartheta(\bN)/2$ that
$$
\varrho_m^o(\cT(2\bN,d)_q,L_p) \ge c(d)\vartheta(\bN)^{1/q-1/p}  .
$$
\end{Lemma}

Let $n$ be a natural number. Denote 
$$
\bH(Q_n)_q := \left\{f\,:\, f \in \cT(Q_n), \quad   \|A_\bs(f)\|_q\le 1\right\}.
$$
Theorem \ref{PT1} implies that $\bH(Q_n)_q$ is embedded in $\bH^r_q C2^{rn}$ with some constant $C$ independent of $n$. Moreover, $\bH(Q_{n+b})_q$ is embedded in $\bH^r_q C(b)2^{rn}$ with some constant $C(b)$ independent of $n$.
We now prove the following analog of Lemma \ref{qpL1}. 

\begin{Lemma}\label{qpL2} Let $1\le q\le p< \infty$, $p>1$ and $n$ be a natural number divisible by $3$.  Denote $S_n:= \min_{\|\bs\|_1=n} |\rho(\bs)|$. Clearly, $S_n\asymp 2^n$. Then there exists a constant $b$ independent of $n$ such that we have for $m\le S_n/2$ 
$$
\varrho_m^o(\bH(Q_{n+b})_q,L_p) \ge c(d)2^{n(1/q-1/p)}  n^{(d-1)/p}  .
$$
\end{Lemma}
\begin{proof} Let a set $\xi\subset \bbT^d := [0,2\pi]^d$ of points $\xi^1,\dots,\xi^m$ be given. Let $n$ be a natural number divisible by $3$ and let $Y_{n,3}$ denote the set of all  $\bs \in \bbN^d$ such that all the coordinates of $\bs$ are natural numbers divisible by $3$ and $\|\bs\|_1=n$. Clearly, $|Y_{n,3}| \asymp n^{d-1}$. Consider the subspaces
$$
T(\xi,\bs) := \{f\in \cT(\rho(\bs)):\, f(\xi^\nu) =0,\quad \nu=1,\dots,m\},\quad \bs \in Y_{n,3}.
$$
Let  $g_{\xi,\bs}\in T(\xi,\bs)$ and a point $\bx_\bs^*$ be such that $|g_{\xi,\bs}(\bx^*)|=\|g_{\xi,\bs}\|_\infty=1$.  
We set $2^{\bs-2} := (2^{s_1-2},\dots,2^{s_d-2})$ and
\be\label{Rf}
 t_\bs(\bx) := g_{\xi,\bs}(\bx)\cK_{2^{\bs-2}}(\bx-\bx_\bs^*),\quad f:= \sum_{\bs\in Y_{n,3}} t_\bs  .
\ee
 Then $f\in \cT(Q_{n+d})$, $f(\xi^\nu)=0$, $\nu=1,\dots,m$, and the bound 
\be\label{R19}
 \|g_{\xi,\bs}\cK_{2^{\bs-2}}\|_q \le \|g_{\xi,\bs}\|_\infty \|\cK_{2^{\bs-2}}\|_q \le C_1(d)2^{n(1-1/q)},\quad \bs \in Y_{n,3}
\ee
implies that for all $\bs$
\be\label{qp1}
\|A_\bs(f)\|_q \le C_1(q,d)2^{n(1-1/q)} .
\ee
In (\ref{R19}) we used the known bound for the $L_q$ norm of the Fej\'er kernel (see \cite{VTbookMA}, p.83, (3.2.7)). 
Moreover, our assumption that all the coordinates of $\bs$ are natural numbers divisible by $3$ implies that 
$$
\delta_\bu(f) = \delta_\bu(t_\bs)
$$
with only one appropriate $\bs$. It is easy to derive from here that for each $\bs$ such that $\|\bs\|_1=n$ there exists 
$\bu(\bs)$ such that $n-2d \le \|\bu(\bs)\|_1 \le n+d$ with the property 
\be\label{qp2}
\|\delta_{\bu(\bs)}(t_\bs)\|_\infty \ge c(d) \|t_\bs\|_\infty, \qquad c(d) >0. 
\ee
 By (\ref{FKm}) we get
\be\label{R20}
|t_\bs(\bx^*)| \ge C_2(d)2^n.
\ee
We now apply the inequality, which directly follows from (\ref{1.2}) of Theorem \ref{T1.1} with $u=p$ and $v=\infty$, and obtain
\be\label{R21}
\|f\|_p \ge C_3(d,p)  2^{n(1-1/p)}n^{(d-1)/p} .
\ee
 
Let $\cM$ be a mapping from $\bbC^m$ to $L_p$. Denote $g_0 := \cM(\mathbf 0)$. Then for $h := f (\max_\bs \|A_
\bs(f)\|_q)^{-1}$ we have 
\be\label{qp3}
\| h-g_0\|_p +\|-h - g_0\|_p \ge 2\|h\|_p.
\ee
Inequality (\ref{qp1}) and the fact that $f\in \cT(Q_{n+d})$ imply  
\be\label{qp4}
\max_\bs \|A_\bs(f)\|_q  \le C_1'(d)2^{n(1-1/q)}\quad \text{and}\quad A_\bs(f) =0, \, \|s\|_1 > n+3d.
\ee
 Relations (\ref{qp3}), (\ref{qp4}), (\ref{R21}), and the fact that both $h$ and $-h$ belong to $\bH(Q_{n+d})_q$ complete the proof of Lemma \ref{qpL2}.

\end{proof}

As a direct corollary of Lemma \ref{qpL2} we obtain Theorem \ref{qpT1} from the Introduction.   
 
\begin{Remark}\label{qpR1} By the Bernstein inequalities (see, for instance, \cite{VTbookMA}, p.89) one finds out that there exists a constant $C(r,d)>0$ such that 
$$
C(r,d) \vartheta(\bN)^{-r} \cT(2\bN,d)_q \subset \bW^r_q.
$$
Then by Lemma \ref{qpL1} we obtain
\be\label{qp5}
\varrho_m^o(\bW^r_q,L_p) \ge c(d)m^{-r +1/q-1/p}.
\ee
It is well known (see, for instance, \cite{DTU}, p.42) that classes $\bW^r_q$ are embedded in the classes $\bH^r_q$.
Therefore, (\ref{qp5}) implies the same lower bound for the classes $\bH^r_q$. However, it is weaker than the bound in 
Theorem \ref{qpT1}. 
\end{Remark}

\section{The case $1\le p\le q\le \infty$}
\label{q1}

We now proceed to the case $1\le p\le q\le \infty$ and concentrate on the special case, when $p=1$ and $q=\infty$. 
It is clear that we have the following inequalities for all $1\le p\le q\le \infty$ and for all classes $\bF^r_q$ ($\bF$ stands for both $\bW$ and $\bH$)
$$
\varrho_m^o(\bF^r_\infty,L_1) \le \varrho_m^o(\bF^r_q,L_p).
$$

The following functions were built in \cite{VT1990} (see also \cite{VTbookMA}, pp. 264--266): For any number $n\in \bbN$ and any set of points $\{\xi^1,\dots,\xi^N\}$, $N\le 2^{n-1}$, there are functions $t_\bs \in \cT(2^{\bs -1},d)$ such that 
$$
t_\bs(\xi^j) =0, \quad j=1,\dots, N, \qquad \|t_\bs\|_\infty \le 1
$$ 
and
\be\label{q11}
\int_{\bbT^d} t(\bx) d\bx \ge c(d)n^{d-1},\qquad t(\bx) := \sum_{\|\bs\|_1=n} t_\bs(\bx).
\ee
Moreover, it was proved there that for $q<\infty$ one has: There exists a constant $c=c(r,q,d)>0$ such that 
$$
 c t 2^{-rn}n^{-(d-1)/2} \in \bW^r_q.
$$
This example implies the following Proposition \ref{q1P1}.

\begin{Proposition}\label{q1P1} For any $q<\infty$ we have for $r>1/q$
$$
\varrho_m^o(\bW^r_q,L_1) \gg m^{-r}(\log m)^{(d-1)/2}.
$$
\end{Proposition}

We now show how the above example implies the lower bound for the $\bH$ classes -- Proposition \ref{q1P2} from Introduction. 
 
{\bf Proof of Proposition \ref{q1P2}.} We prove that there is a positive constant $c(d,q,r)$ such that 
$$
c(d,q,r)t2^{-rn} \in \bH^r_\infty.
$$
 For that we estimate $\|A_\bu(t)\|_\infty$ for all $\bu$ and use Theorem \ref{PT1}. Obviously, $A_\bu(t)=0$ if for some $j$ we have $2^{u_j-2} > 2^{s_j-1}$. Therefore, 
it is sufficient to analyze $\bu$ such that $\|\bu\|_1 \le n+d$. In the same way we see that 
$$
A_\bu(t) = \sum_{\bs\in Y_n(\bu)} A_\bu(t_\bs),\quad Y_n(\bu):=\{\bs\,:\,  s_j\ge u_j-1,\, j=1,\dots, d, \, \|\bs\|_1=n\}.
$$
Denote $w_j:= s_j-u_j+1$, $\bw:=(w_1,\dots,w_d)$. Then for $\bs\in Y_n(\bu)$ we have $w_j\ge 0$, $j=1,\dots,d$, and 
$$
\|\bw\|_1 = \|\bs\|_1-\|\bu\|_1 +d =n-\|\bu\|_1 +d.
$$
The total number of such $\bw$ is $\ll (n-\|\bu\|_1 +d)^{d-1}$. Therefore,
$$
 \|A_\bu(t)\|_\infty \ll (n-\|\bu\|_1 +d)^{d-1} \ll 2^{r(n-\|\bu\|_1)}.
 $$
By Theorem \ref{PT1} this implies that there exists a positive constant $c(d,q,r)$ such that  $c(d,q,r)t2^{-rn} \in \bH^r_\infty$.
We now use (\ref{q11}) and complete the proof in the same way as it has been done in the proof of Lemma \ref{qpL2} above. 
 
We now make a comment in the style of Remark \ref{qpR1}, which points out that somewhat weaker than Propositions \ref{q1P1} and \ref{q1P2}  results can be derived from the known results. The following Lemma \ref{q1L1} was proved in \cite{VT202}.

\begin{Lemma}[{\cite{VT202}}]\label{q1L1} Let $\cT(\bN,d)_\infty$ denote the unit $L_\infty$-ball of the subspace $\cT(\bN,d)$. Then we have for $m\le \vartheta(\bN)/2$ that
$$
\varrho_m^o(\cT(\bN,d)_\infty,L_1) \ge c(d) >0  .
$$
\end{Lemma}

\begin{Remark}\label{q1R1} In the same way as above in Remark \ref{qpR1} by using Lemma \ref{q1L1} instead of Lemma \ref{qpL1} we obtain
$$
\varrho_m^o(\bH^r_\infty,L_1) \gg \varrho_m^o(\bW^r_\infty,L_1) \gg m^{-r}  .
$$

\end{Remark}

{\bf Comment on the Gelfand width. \ref{q1}.1.}  It is easy to see from the construction of functions $t$ in \cite{VT1990} (see also \cite{VTbookMA}, pp. 264--266), mentioned above, that the sampling linear functionals can be replaced by any linear functionals. This means that Propositions \ref{q1P1} and \ref{q1P2} hold for
the following asymptotic characteristics as well. The Gelfand width is defined as follows
$$
c_m(\bF,X) := \inf_{\ff_1,\dots,\ff_m}\sup_{f\in \bF:\, \ff_j(f)=0,\,j=1,\dots,m} \|f\|_X
$$
where $\ff_1,\dots,\ff_m$ are linear functionals. 
Thus,   we have
\be\label{Ge0}
c_m(\bW^r_\infty,L_1) \gg   m^{-r}(\log m)^{(d-1)/2}
\ee
and
\be\label{Ge1}
c_m(\bH^r_\infty,L_1) \gg   m^{-r}(\log m)^{d-1}.
\ee

\section{Sampling recovery on classes with structural condition}
\label{S}
 
 We formulate the setting of the sampling recovery in the general space $L_p(\Omega,\mu)$, $1\le p<\infty$, and instead of the trigonometric system $\cT^d$ we study a general uniformly bounded system $\Psi = \{\psi_{\bk}\}_{\bk\in \bbZ^d}$. More precisely, let $\Omega$ be a compact subset of $\bbR^d$ with the probability measure $\mu$ on it. By the $L_p$ norm, $1\le p< \infty$, of the complex valued function defined on $\Omega$,  we understand
$$
\|f\|_p:=\|f\|_{L_p(\Omega,\mu)} := \left(\int_\Omega |f|^pd\mu\right)^{1/p}\quad \text{and}\quad \|f\|_\infty := \sup_{\bx\in\Omega}|f(\bx)|.
$$
Let a uniformly bounded system $\Psi:=\{\psi_\bk\}_{\bk\in \bbZ^d}$ be indexed by $\bk\in\bbZ^d$. Consider a sequence of subsets $\cG:=\{G_j\}_{j=1}^\infty$, $G_j \subset \bbZ^d$, $j=1,2,\dots$, such that
 \be\label{Di4}
 G_1\subset G_2\subset \cdots \subset G_j\subset G_{j+1} \subset \cdots,\qquad \bigcup_{j=1}^\infty G_j =\bbZ^d.
 \ee
Consider functions representable in the form of absolutely convergent series 
\be\label{Direpr}
f = \sum_{\bk\in\bbZ^d} a_\bk(f)\psi_\bk,\qquad \sum_{\bk\in\bbZ^d} |a_\bk(f)|<\infty.
\ee
For $\bt \in (0,1]$ and $r>0$ consider the following class $\bA^r_\bt(\Psi,\cG)$ of functions $f$, which have representations (\ref{Direpr}) satisfying   conditions
\be\label{DiAr}
  \left(\sum_{  \bk\in G_j\setminus G_{j-1}} |a_\bk(f)|^\bt\right)^{1/\bt} \le 2^{-rj},\quad j=1,2,\dots,\quad G_0 :=\emptyset  .
\ee
Probably, the first realization of the idea of the classes $\bA^r_\bt(\Psi,\cG)$ was realised in \cite{VT150} in the special case, when $\Psi$ is the trigonometric system $\cT^d := \{e^{i(\bk,\bx)}\}_{\bk\in\bbZ^d}$ (see \cite{VT203} for a detailed historical discussion). The classes $\bA^r_\bt(\Psi)$ studied in \cite{VT203} correspond to the case of $\bA^r_\bt(\Psi,\cG)$ with
\be\label{Di8}
G_j:= \{\bk\in \bbZ^d\,:\, \|\bk\|_\infty <2^j\},\quad j=1,2,\dots.
\ee

We now define classes $\bW^{a,b}_{A_\bt}(\Psi)$, which were introduced and studied in \cite{VT205}. 
For 
\be\label{R9a}
f=\sum_{\bk}a_\bk(f) \psi_\bk,\qquad \sum_{\bk}|a_\bk(f) | <\infty,
\ee
 denote
$$
\delta_\bs(f,\Psi):= \sum_{\bk\in \rho(\bs)}a_\bk(f) \psi_\bk,\quad f_j:=\sum_{\|\bs\|_1=j}\delta_\bs(f,\Psi), \quad j\in \bbN_0,\quad \bbN_0:=\bbN\cup \{0\}
$$
and   for $\bt\in (0,1]$
$$
|f|_{A_\bt(\Psi)} := \left(\sum_{\bk}|a_\bk(f)|^\bt\right)^{1/\bt}.
$$
Note, that if representations (\ref{R9a}) are unique, then in the case $\bt=1$ the characteristic $|f|_{A_\bt(\Psi)}$ is the norm and in the case $\bt \in (0,1)$ it is a quasi-norm.   
For parameters $ a\in \bbR_+$, $ b\in \bbR$ define the class $\bW^{a,b}_{A_\bt}(\Psi)$ of functions $f$ for which 
there exists a representation (\ref{R9a}) satisfying
\be\label{R10}
 |f_j|_{A_\bt(\Psi)} \le 2^{-aj}(\bar j)^{(d-1)b},\quad \bar j:=\max(j,1), \quad j\in \bbN_0.
\ee

In the case, when $\Psi$ is the trigonometric system and $\bt=1$, classes $\bW^{a,b}_{A_\bt}(\Psi)$
were introduced in \cite{VT150}. The general definition in the case $\bt=1$ is given in \cite{DTM2}. We use the notation $A$ in place of $A_1$. Thus, $\bW^{a,b}_{A}(\Psi):= \bW^{a,b}_{A_1}(\Psi)$. Note that the $\bW^{a,b}_{A_\bt}(\Psi)$ classes can be seen as the $\bW$ type classes with the structural condition on the coefficients in the quasi-norm $A_\bt$. 

We now define an analog of the $\bW^{a,b}_{A_\bt}(\Psi)$ classes in a style of the $\bH$ classes. For parameters $ a\in \bbR_+$, $ b\in \bbR$ define the class $\bH^{a,b}_{A_\bt}(\Psi)$ of functions $f$ for which 
there exists a representation (\ref{R9a}) satisfying
\be\label{R11}
 |\delta_\bs(f,\Psi)|_{A_\bt(\Psi)} \le 2^{-aj}(\bar j)^{(d-1)b},\quad \bar j:=\max(j,1), \quad j\in \bbN_0,\quad \|\bs\|_1=j.
\ee

Note, that the following embedding follows directly from the definitions of the classes  
\be\label{emb}
\bH^{a,b}_{A_\bt} \hookrightarrow \bW^{a,b'}_{A_\bt} \quad \text{with} \quad   b'=b+1/\bt.
\ee

We will need some simple properties of the quasi-norms $|\cdot|_{A_\bt}$. 

\begin{Proposition}\label{SP1} Assume that $\Psi$ is a uniformly bounded ($\|\psi_\bk\|_\infty \le B$, $\bk\in\bbZ^d$)
orthonormal system. Denote for $\bN = (N_1,\dots,N_d)$, $N_j\in \bbN_0$, $j=1,\dots,d$, 
$$
\Psi(\bN,d) := \left\{f=\sum_{\bk\in \Pi(\bN,d)} c_\bk \psi_\bk\right\},\quad \vartheta(\bN) :=\prod_{j=1}^d (2N_j+1). 
$$
Then for $q\in [1,2]$ we have for $f\in \Psi(\bN,d)$
$$
|f|_A \le B^{2/q-1} \vartheta(\bN)^{1/q}\|f\|_q.
$$
\end{Proposition}
\begin{proof} Denote for an array $\bv = \{v_\bk\}_{\bk\in \Pi(\bN,d)}$, $|v_\bk|=1$, $\bk\in \Pi(\bN,d)$,
$$
D_{\bN,\Psi}(\bv,\bx) := \sum_{\bk\in \Pi(\bN,d)} v_\bk \psi_\bk(\bx).
$$
Let $f$ have the representation (\ref{R9a}). By the orthonormality assumption we have
\be\label{S1}
|f|_A = \sum_{\bk\in \Pi(\bN,d)} |a_\bk(f)| = \int_\Omega  D_{\bN,\Psi}(\bv,\bx) {\bar f} d\bx,
\ee
where $v_\bk := \sign a_\bk(f) := a_\bk(f)/|a_\bk(f)|$, if $a_\bk \neq 0$ and $v_\bk =1$ if $a_\bk = 0$; 
${\bar f}$ is the complex conjugate to $f$. From (\ref{S1}) we obtain
\be\label{S2}
|f|_A \le \|D_{\bN,\Psi}(\bv,\cdot)\|_{q'}\|\bar f\|_q,\quad q' := q/(q-1).
\ee
Using simple relations
$$
\|D_{\bN,\Psi}(\bv,\cdot)\|_2 = \vartheta(\bN)^{1/2},\qquad \|D_{\bN,\Psi}(\bv,\cdot)\|_\infty \le B \vartheta(\bN)
$$
and the well known inequality $\|g\|_p \le \|g\|_2^{2/p}\|g\|_\infty^{1-2/p}$, $p\in [2,\infty)$, we get
\be\label{S3}
 \|D_{\bN,\Psi}(\bv,\cdot)\|_{q'} \le B^{2/q-1} \vartheta(\bN)^{1/q}.
\ee
Combining (\ref{S2}) and (\ref{S3}), we complete the proof.

\end{proof}

Let $\bt\in (0,1)$. Then for any set of numbers $\{y_j\}_{j=1}^M$ we have by the H{\"o}lder inequality
\be\label{S4}
\sum_{j=1}^M |y_j|^\bt \le \left(\sum_{j=1}^M |y_j|\right)^{\bt} M^{1-\bt}.
\ee

Inequality (\ref{S4}) and Proposition \ref{SP1} imply the following Corollary \ref{SC1}.

\begin{Corollary}\label{SC1} Under assumptions of Proposition \ref{SP1} for $q\in [1,2]$ and $\bt\in(0,1]$  we have for $f\in \Psi(\bN,d)$
$$
|f|_{A_\bt} \le \vartheta(\bN)^{1/\bt-1} |f|_A \le B^{2/q-1} \vartheta(\bN)^{1/\bt -1 +1/q}\|f\|_q.
$$
\end{Corollary}

We now proceed to the main result of this section -- the lower bound for $\rho^o_m(\bH^{a,b}_{A_\bt},L_p)$. 

\begin{Theorem}\label{ST1} Let $a>0$ and $b\in \bbR$. Then for $\bt\in (0,1]$ and $p\in [2,\infty)$
$$
\rho^o_m(\bH^{a,b}_{A_\bt}(\cT^d),L_p) \gg m^{1-1/p-1/\bt-a}(\log m)^{(d-1)(b+1/p)}.
$$
\end{Theorem}
\begin{proof} Let $f$ be the function defined by (\ref{Rf}) in the proof of Lemma \ref{qpL2}. Then by (\ref{R19}) and 
Corollary \ref{SC1} we obtain
\be\label{S5}
|t_\bs|_{A_\bt} \ll  2^{n/\bt},
\ee
which easily implies that there exists a positive constant  $c$ independent of $n$ such that $c2^{n(-a-1/\bt)} n^{(d-1)b}f \in 
\bH^{a,b}_{A_\bt}$. We now use (\ref{R21}) and complete the proof in the same way as it has been done in the proof of 
Lemma \ref{qpL2}. 
\end{proof}

The following upper bound was proved in \cite{VT205} for $p\in [2,\infty)$ 
\begin{equation}\label{Mtr}
 \varrho_{m}^{o}(\bW^{a,b}_{A_\bt}(\cT^d),L_p(\bbT^d)) \ll   \left(\frac{m}{(\log m)^3}\right)^{1-1/p -1/\bt -a}  (\log(m))^{(d-1)(a+b)} .
\end{equation}

Note that Theorem \ref{ST1} does not cover the case $p\in [1,2)$. We now present the corresponding lower bound in the case $p=1$. 
In the same way as we derived Theorem \ref{ST1} from the example built in the proof of Lemma \ref{qpL2} we can derive the following lower bound from the example presented in Section~\ref{q1}
\be\label{S7}
\rho^o_m(\bH^{a,b}_{A_\bt}(\cT^d),L_1) \gg m^{1/2-1/\bt-a}(\log m)^{(d-1)(b+1)}.
\ee
This bound, the upper bound (\ref{Mtr}) with $p=2$, and the embedding (\ref{emb}) show that the characteristics 
$\rho^o_m(\bH^{a,b}_{A_\bt}(\cT^d),L_p)$ have the same power decay $m^{1/2-1/\bt-a}$ for all $p\in [1,2]$. 
 
 {\bf Comment \ref{S}.1.} Similarly to Comment \ref{q1}.1 we have the following analog of the bound (\ref{S7})
\be\label{S8}
c_m(\bH^{a,b}_{A_\bt}(\cT^d),L_1) \gg m^{1/2-1/\bt-a}(\log m)^{(d-1)(b+1)}.
\ee

\section{Discussion}
\label{D}

In this paper we focus on the study of the lower bounds for the nonlinear characteristic  $\varrho_m^o(W,L_p)$. 
Most of the known results on optimal sampling recovery deal with the linear recovery methods. Recall the setting 
 of the optimal linear recovery. For a fixed $m$ and a set of points  $\xi:=\{\xi^j\}_{j=1}^m\subset \Omega$, let $\Phi $ be a linear operator from $\bbC^m$ into $L_p(\Omega,\mu)$.
Denote for a class $W$ (usually, centrally symmetric and compact subset of $L_p(\Omega,\mu)$) 
$$
\varrho_m(W,L_p) := \inf_{\text{linear}\, \Phi; \,\xi} \sup_{f\in W} \|f-\Phi(f(\xi^1),\dots,f(\xi^m))\|_p.
$$
The characteristic $\varrho_m(W,L_p)$ was introduced and studied in \cite{VT51}. The reader can find a detailed discussion of results on $\varrho_m(W,L_p)$ in the books \cite{DTU} and \cite{VTbookMA}. Recently, a substantial progress in estimating $\varrho_m(W,L_p)$ and $\varrho_m^o(W,L_p)$ (mostly, the case of recovery in the $L_2$ norm was studied) was made in the papers \cite{KU}, \cite{KU2},
\cite{NSU}, \cite{VT183},  \cite{TU1},  
\cite{JUV}, \cite{DKU},  \cite{DTM1}, \cite{DTM2}, \cite{DTM3},  
   \cite{VT203}, \cite{KoTe}, \cite{VT205}. 
   
We have an obvious inequality
\be\label{D1}
\varrho_m^o(W,L_p) \le \varrho_m(W,L_p),
\ee   
which means that the upper bounds for $\varrho_m(W,L_p)$ serve as upper bounds for $\varrho_m^o(W,L_p)$ and 
the lower bounds for $\varrho_m^o(W,L_p)$ serve as lower bounds for $\varrho_m(W,L_p)$. It is an interesting problem to understand for which function classes $W$ the rates of decay of the characteristics $\varrho_m(W,L_p)$ and $\varrho_m^o(W,L_p)$ coincide. It is known that even in the classical case of classes $\bW^r_q$, $1<q<2$, the rates of 
$\varrho_m(\bW^r_q,L_2)$ and $\varrho_m^o(\bW^r_q,L_2)$ do not coincide. It was observed in \cite{JUV} that in the case $1<q<2$ for large enough $d$ the upper bound for $\varrho_m^o(\bW^r_q,L_2)$ and the known lower bound for $\varrho_m(\bW^r_q,L_2)$ imply that $\varrho_m^o(\bW^r_q,L_2) =o(\varrho_m(\bW^r_q,L_2))$. This means that in those cases nonlinear methods give better rate of  decay of errors of sampling recovery than linear methods do. We now discuss this important phenomenon in detail. 
We begin with the upper bounds (\ref{H6}) and (\ref{H7}) for $\varrho_m^o(\bW^r_q,L_2)$ given in the Introduction.

 For instance, (\ref{H7}) and (\ref{H8}) imply
\be\label{H8a}
\varrho_{m}^{o}(\bW^{r}_q,L_2(\bbT^d)) \ll  \left(\frac{m}{(\log m)^3}\right)^{-r+1/q-1/2} (\log m)^{(d-1)(r+1-2/q)}  .    
\end{equation}

We now proceed to the lower bounds for $\varrho_{m}(\bW^{r}_q,L_2(\bbT^d))$. It is obvious that 
\be\label{Kol}
d_m(W,L_p) \le \varrho_{m}(W,L_p),
\ee
where $d_m(W,X)$ is the Kolmogorov width of $W$ in a Banach space $X$:
$$
d_m(W,X) := \inf_{Y\subset X, \dim Y\le m} \sup_{f\in W} \inf_{y\in Y} \|f-y\|_X.
$$
Here are the known bounds on the $d_m (\bW^r_q, L_2)$ (see, for instance, \cite{VTbookMA}, p.216): For $1<q\le 2$, $r > 1/q-1/2$ and $2<q<\infty$, $r> 0$ 
\be\label{H6a}
d_m(\bW^r_q,L_2) \asymp \left(\frac{(\log m)^{d-1})}{m}\right)^{r-(1/q-1/2)_+},\quad (a)_+ := \max(a,0).
\ee
Taking into account that $r+1-2/q < r+1/2-1/q$ for $1<q<2$ we conclude from (\ref{H8a}), (\ref{Kol}), and (\ref{H6a}) that 
for large enough $d$ we have $\varrho_m^o(\bW^r_q,L_2) =o(\varrho_m(\bW^r_q,L_2))$.

It is interesting to point out that we do not know if the above effect, which holds for the $\bW$ classes, holds for the $\bH$ classes as well. The following upper bound is known (see, for instance, \cite{VTbookMA}, p.308) for the linear recovery in the case $1\le q<p <\infty$, $r >1/q$,
\be\label{D6}
\varrho_{m}(\bH^{r}_q,L_p(\bbT^d)) \ll  \left(\frac{m}{(\log m)^{d-1}}\right)^{-r+1/q-1/p} (\log m)^{(d-1)/p}.
\ee
In the case $p\ge2$ the bound (\ref{D6}) is better than the bound (\ref{In6}) from Introduction. This means that  in the case $1< q \le 2 \le p<\infty$ the known upper bounds for the linear recovery are better than those in (\ref{In6}).

Let us make some comments on the technique available for proving the lower bounds for the optimal recovery. In the case 
of linear recovery we can use the inequality (\ref{Kol}) or even a stronger one with the Kolmogorov width replaced by 
the linear width. This theory is well developed (see, for instance, the books \cite{DTU} and \cite{VTbookMA}). 
In the case of nonlinear recovery we cannot use the above approach. It seems like the only technique, which is available now, is based on the following simple observation. 

 \begin{Proposition}\label{DP1} Let $W\subset X$ be a symmetric ($f\in W \Rightarrow -f \in W$) subset, consisting of continuous functions. Then
$$
\varrho_m^o(W,X) \ge \inf_{\xi^1,\dots,\xi^m} \sup_{f\in W\,:\, f(\xi^j)=0,\, j=1,\dots,m} \|f\|_X.
$$
\end{Proposition}

{\bf Acknowledgements.}
This work was supported by Ministry of Science and Higher Education of 
the Russian Federation (Grant No. 075-15-2024-529).

  \Addresses


\begin{thebibliography}{9999}
 
 
 


\bibitem{DKU} M. Dolbeault , D. Krieg, and M. Ullrich, A sharp upper bound for sampling numbers in $L_2$, {\it Appl. Comput. Harmon. Anal.} {\bf  63} (2023), 113--134;
arXiv:2204.12621v1 [math.NA] 26 Apr 2022.

 

 \bibitem{DTM1} F.  Dai and V.N.  Temlyakov, Universal discretization and sparse sampling recovery,
 arXiv:2301.05962v1 [math.NA] 14 Jan 2023.
 
 \bibitem{DTM2} F. Dai and V.N. Temlyakov, Random points are good for universal discretization, J. Math. Anal. Appl., {\bf 529} (2024) 127570; arXiv.2301.12536[math.FA] 5 Feb 2023.
 
\bibitem{DTM3} F. Dai and V.N. Temlyakov, Lebesgue-type inequalities in sparse sampling recovery,
arXiv:2307.04161v1 [math.NA] 9 Jul 2023.

 \bibitem{DTU} Dinh D{\~u}ng, V.N. Temlyakov, and T. Ullrich, Hyperbolic Cross Approximation, Advanced Courses in Mathematics CRM Barcelona, Birkh{\"a}user, 2018; arXiv:1601.03978v2 [math.NA] 2 Dec 2016.
  
 
  
 
 \bibitem{JUV} T. Jahn, T. Ullrich, and F. Voigtlaender, Sampling numbers of smoothness classes via 
 $\ell^1$-minimization, arXiv:2212.00445v3 [math.NA] 31 Jul 2023. 
 
 
 
 
 

\bibitem{KoTe} E. Kosov and V. Temlyakov, Bounds for the sampling discretization error and their applications to universal sampling discretization, arXiv:2312.05670v2 [math.NA] 27 Jan 2024.
 
 \bibitem{KU} D. Krieg and M. Ullrich, Function values are enough for $L_2$-ap\-pro\-ximation, {\it Found. Comp. Math.}, doi:10.1007/s10208-020-09481-w;
arXiv:1905.02516v4 [math.NA] 19 Mar 2020.

\bibitem{KU2} D. Krieg and M. Ullrich, Function values are enough for $L_2$-approximation: Part II, {\it J. Complexity}, doi:10.1016/j.jco.2021.101569; arXiv:2011.01779v1 [math.NA] 3 Nov 2020.
 
 

 
 \bibitem{NSU} N. Nagel, M. Sch{\"a}fer, T. Ullrich, A new upper bound for sampling numbers, {\it Found. Comp. Math.}, Pub Date: 2021-04-26, DOI: 10.1007/s10208-021-09504-0;
arXiv:2010.00327v1 [math.NA] 30 Sep 2020. 


\bibitem{Smol} S.A. Smolyak,   Quadrature and interpolation formulas for tensor products of certain classes of functions. Dokl. Akad. Nauk SSSR, 148 (1963), 1042–1045; English translation in Soviet Math. Dokl., 4 (1963).

\bibitem{VT205} A.P. Solodov and V.N. Temlyakov, Sampling recovery on function classes with a structural condition, arXiv:2404.07210v2 [math.NA] 25 Jun 2024.

\bibitem{T29} V.N. Temlyakov, Approximation of Periodic Functions of Several Variables by Bilinear Forms, Izvestiya AN SSSR, Ser. Mat., {\bf 50} (1986), 137--155; English transl. in Mathematics of the USSR-Izvestia, {\bf 28} (1987), 133--150. 

\bibitem{VTmon} V.N. Temlyakov, Approximation of functions with bounded mixed derivative, Trudy MIAN, {\bf 178} (1986), 1--112. English transl. in Proc. Steklov Inst. Math., {\bf 1} (1989).

\bibitem{VT1990}  V.N. Temlyakov,  On a way of obtaining lower estimates 
for the errors of quadrature formulas, Matem. Sbornik, {\bf 181} (1990),
 1403--1413;  English transl. in  Math. USSR Sbornik, {\bf 71} (1992). 


\bibitem{VT51} V.N. Temlyakov, On Approximate Recovery of Functions with Bounded Mixed Derivative, J. Complexity, {\bf 9} (1993), 41--59.
 
 
 
  
 \bibitem{VT150} V.N. Temlyakov, Constructive sparse trigonometric approximation and other problems for functions with mixed smoothness, arXiv: 1412.8647v1 [math.NA] 24 Dec 2014, 1--37; Matem. Sb., {\bf 206} (2015), 131--160. 
 
   
 
 \bibitem{VTbookMA} V. Temlyakov, {\em Multivariate Approximation}, Cambridge University Press, 2018.

\bibitem{VT183} V.N. Temlyakov, On optimal recovery in $L_2$, J. Complexity {\bf 65} (2021), 101545; arXiv:2010.03103v1 [math.NA] 7 Oct 2020.

\bibitem{VT202} V. Temlyakov, Sparse sampling recovery by greedy algorithms, arXiv:2312.13163v2 [math.NA] 30 Dec 2023.

\bibitem{VT203} V. Temlyakov, Sparse sampling recovery in integral norms on some function classes, arXiv:2401.14670v1 [math.NA] 26 Jan 2024.
 
 \bibitem{TU1} V.N. Temlyakov and T. Ullrich, Bounds on Kolmogorov widths of classes with small mixed smoothness, J. Complexity, Available online 4 May 2021, 101575; arXiv:2012.09925v1 [math.NA] 17 Dec 2020.

\bibitem{TWW} J.F. Traub, G.W. Wasilkowski, and H. Wo{\'z}niakowski, Information-Based Complexity, Academic Press, Inc., 1988.


 
  \end{thebibliography}
\end{document}